\documentclass{amsart}
\usepackage{graphicx}
\usepackage[all]{xy}

\usepackage{amsmath}
\usepackage{amssymb}
\usepackage{amsthm}
\usepackage{tikz-cd}
\theoremstyle{plain}
\newtheorem{thm}{Theorem}[section]
\newtheorem{lemma}[thm]{Lemma}

\newtheorem{corollary}[thm]{Corollary}

\theoremstyle{definition}
\newtheorem{definition}[thm]{Definition}
\theoremstyle{remark}
\newtheorem{remark}[thm]{Remark}

\DeclareMathOperator{\cofib}{cofib}

\DeclareMathOperator{\vir}{vir}
\DeclareMathOperator{\id}{id}
\DeclareMathOperator{\pt}{pt}
\DeclareMathOperator{\Lag}{Lag}
\DeclareMathOperator{\qsp}{qsp}

\begin{document}

\title{Lagrangian correspondences and pullbacks of virtual fundamental classes} 
\author{Timo Sch\"urg}
\address{Dep. of Mathematics and Natural Sciences, Darmstadt University of Applied Sciences, Germany}
\email{timo.schuerg@h-da.de}

\maketitle

\begin{abstract}
    We argue that Lagrangian correspondences are the correct framework to study
    functoriality of virtual fundamental classes arising from a $-2$-symplectic
    derived structure.
\end{abstract}

\section{Introduction}

In any oriented Borel-Moore homology theory $A_*$, we can perform pullbacks of homology classes along a
local complete intersection morphism. Such a pullback operation along an l.c.i.
morphism $f\colon X \to Y$ will map the fundamental class $[Y] \in A_*(Y)$ to the
fundamental class $[X] \in A_*(X)$.

The story becomes more complicated when studying the virtual fundamental classes
introduced in \cite{behrend}.  These classes live in the Chow group of $X$. They
depend on an additional structure which is not inherent to the scheme structure.
This additional datum is a two term perfect obstruction theory, a morphism
$\mathbb{E} \to L_X$ in the derived category of $X$ subject to finiteness
conditions.

Since this additional information is not encoded in the scheme structure, a morphism $f
\colon X \to Y$ of schemes endowed with perfect obstruction theories need
not map the virtual fundamental class of $Y$ to the virtual class of $X$, even
if $f$ induces a pull back operation from the Chow groups of $Y$ to the Chow
groups of $X$.

The correct compatibilities between schemes with perfect obstruction
theories to ensure the virtual fundamental class pulls back as expected were
identified in \cite{manolache}. If we find a compatible relative perfect
obstruction theory sitting in a diagram
\[
    \begin{tikzcd}
        f^*\mathbb{E}_Y \ar[r] \ar[d] & \mathbb{E}_X \ar[r] \ar[d] & \mathbb{E}_f \ar[d] \\
        f^*L_Y \ar[r] & L_X \ar[r] & L_f
    \end{tikzcd}
\]
we can ensure the virtual fundamental class of $Y$ gets pulled back to the
virtual fundamental class of $X$.

The compatibility Manolache identified becomes natural in the setting of
quasi-smooth morphisms of quasi-smooth derived schemes in derived algebraic
geometry. Every such morphism automatically carries the necessary compatibility
condition, and maps the virtual fundamental class of $Y$ to the virtual
fundamental class of $X$.

This makes the case that studying virtual fundamental classes is best done in
the category of derived schemes with quasi-smooth morphisms.

This notion can be made axiomatic using oriented Borel-Moore functors with
quasi-smooth pullbacks. These are homology theories with pull-back or Gysin
operations along quasi-smooth morphisms. Using this formalism, the virtual class
of a quasi-smooth derived scheme with structure morphism $\pi \colon X \to \pt$ is simply $[X]^{\vir} =
\pi_X ^* ([1])$ and every quasi-smooth morphism of quasi-smooth schemes
automatically satisfies $f^*([Y]^{\vir}) = [X]^{\vir}$.

A new type of virtual fundamental classes were introduced in
\cite{borisov} and \cite{ohthomas}. These do not arise from a two term perfect
obstruction theory, but from a three term perfect obstruction theory with an
additional symmetry. This additional structure again has its origin in derived
algebraic geometry. Every $-2$-shifted symplectic derived scheme \cite{shifted}
naturally induces the necessary structure on its underlying classical scheme.

This leads to the question studied in this paper: Just as quasi-smooth morphisms
are the correct setting to study pull-backs of virtual fundamental classes
arising from perfect obstruction theories, what
is the correct setting for the virtual fundamental classes arising from
$-2$-symplectic derived structures?

As in the case of the virtual fundamental classes of Behrend and Fantechi, a
morphism $f\colon X \to Y$ of schemes need not respect this additional structure. The
question which additional compatibilities are necessary was studied in
\cite{park}. Park shows that the diagram
\[
    \begin{tikzcd}
        \mathbb{D}^{\vee}[2] \arrow[r] \arrow[d] &
        \mathbb{E}_X \arrow[r] \arrow[d] & \mathbb{E}_f \arrow[d] \\
        f^* \mathbb{E}_Y \arrow[r] \arrow[d] & \mathbb{D} \arrow[r] \arrow[d] &
        \mathbb{E}_f \arrow[d]\\ f^* L_Y \arrow[r] & L_X \arrow[r] & L_f \\
    \end{tikzcd}
\]
ensures that the virtual fundamental class of $Y$ gets pulled back to the
virtual fundamental class of $X$.

Note that here something new shows up, which was not present in the case of two
term perfect obstruction theories: the complex $\mathbb{D}$.

The aim of this note is to make the case that Lagrangian correspondences are the
right setting to study functoriality properties of the Oh-Thomas virtual
classes. To justify this approach, we show some properties of oriented
Borel-Moore functors with quasi-smooth pullbacks on this category and identify
the necessary compatibilities to ensure the Oh-Thomas virtual class behaves just
like a fundamental class.

We then specialize to giving a geometric interpretation of the complex
$\mathbb{D}$ popping up in Park's pullback compatibility diagram. The geometric
interpretation  should be a space $\Gamma$ having $\mathbb{D}$ as cotangent
complex. We introduce the notion of a Lagrangian graph with \'etale section of a
morphism of shifted symplectic derived schemes and verify that the cotangent complex of
such a graph gives exactly the compatibilities identified by Park in
\cite{park}.

\subsection{Related Works}

The idea that Lagrangian correspondences are the correct setting to study
virtual fundamental classes arising from -2-symplectic derived structures seems
to be well known in the community. It is explicitly stated in
\cite[Remark 5.2.5]{parkpushforward}, which I was not aware of when I started
this project. I nevertheless hope this paper can provide some benefit.

\subsection{Overview}

After recalling the virtual fundamental class of Oh and Thomas in Section
\ref{vfc}, we study Lagrangian correspondences and their fundamental classes in
an arbitrary oriented Borel-Moore homology theory in section \ref{corr}. This
section aims to make the case that Lagrangian correspondences are key to
obtaining compatibilities between virtual fundamental classes of shifted
symplectic derived schemes. In the final section \ref{hf} we introduce the
notion of Lagrangian graph with \'etale section and verify such a graph induces Park's compatibility
diagram.

\subsection{Acknowledgements}
I would like to thank Hyeonjun Park for helpful comments on an early draft.

\section{The virtual fundamental class of Oh-Thomas}\label{vfc}

In this section we give a brief summary of \cite{ohthomas}, in particular where
the derived structure comes into play.

Given a $(-2)$-symplectic derived scheme $X$ of virtual dimension $2n$, Oh and
Thomas construct a virtual fundamental class of dimension $n$ on the classical
scheme $t_0(X)$ underlying $X$.

To produce this class, they pull back the cotangent complex of $X$ along the
inclusion $j \colon t_0(X) \to X$ of the classical scheme underlying $X$. This
gives a morphism of complexes
\[
    \mathbb{E} := j^* L_X \to L_{t_0(X)}
\]
They
resolve the three-term complex $\mathbb{E}$ as a self-dual complex of locally
free sheaves $T \to E \to T^*$. The symplectic form on $X$ induces a quadratic
form $q$ on the middle term $E$, making $E$ an orthogonal bundle.

The key point where the derived structure enters is showing that the pull back
of the Behrend-Fantechi normal cone gives an \emph{isotropic} cone in $E$. This
crucially depends on the Darboux theorems of \cite{BBBJ}, which in turn depends
crucially on the $(-2)$-symplectic structure on $X$.

The virtual fundamental class is then the localized square-root Euler class of
the isotropic cone $C$ in $E$, giving a class in $A_n(t_0(X), \mathbb{Z}[\frac 1
2])$.

\section{Lagrangian correspondences}\label{corr}

We assume the reader is acquainted with the basics of derived algebraic geometry
\cite{toenems, HAGII}, including the notion of $n$-shifted symplectic structure
\cite{shifted}.

Since the notion of a Lagrangian correspondence might not be so well known, we
recall the definition.

\begin{definition}
    Let $(X, \omega)$ and $(Y,\sigma)$ be $n$-shifted symplectic derived
    schemes. A \emph{Lagrangian correspondence} from $X$ to $Y$ is a span
    \[
        \begin{tikzcd}
            & M \ar[dl, "p" swap] \ar[dr, "q"] & \\
            X & & Y
        \end{tikzcd}
    \]
    inducing a homotopy cartesian square of complexes 
    \[
        \begin{tikzcd}
            T_M \ar[r] \ar[d]& p^* T_X \ar[d]\\
            q^* T_Y \ar[r]& L_M [n] .
        \end{tikzcd}
    \]
    on $M$.
\end{definition}
\begin{remark}\label{rem:AlternativeLagrange}
    An equivalent definition is that $M \to X \times \bar{Y} $ is a Lagrangian
    morphism.  Here $\bar{Y}$ is $Y$ equipped with the negative symplectic form
    $- \sigma$.
\end{remark}

A special case of Lagrangian correspondences is given by symplectomorphisms. The
following definition is taken from \cite{benbassat}.

\begin{definition}
    A weak equivalence $f \colon (X,\omega) \to
    (Y, \sigma)$ of derived schemes is a \emph{symplectomorphism} if
    \[
        \begin{tikzcd}
            & X \ar[dl, "id" swap] \ar[dr, "f"] & \\
            X \ar[rr, "f"]& & Y
        \end{tikzcd}
    \]
    is a Lagrangian correspondence.
\end{definition}

\begin{remark}
    Using Remark \ref{rem:AlternativeLagrange}, $f$ is a symplectomorphism if
    the graph morphism $\Gamma_f \colon X \to X \times \bar{Y}$ is a Lagrangian
    morphism.
\end{remark}

In \cite{haugseng} it is shown that there exists a nice $\infty$-category
$\Lag^n_{(\infty,1)}$ of $n$-symplectic Lagrangian correspondences. An important
ingredient is a theorem proved in \cite{calaque}, showing that two Lagrangian
correspondences can be composed.

%
%

We now want to relate the functoriality properties of the virtual fundamental
class to Lagrangian correspondences. To establish the relationship, we introduce
a subcategory of $\Lag^n_{(\infty, 1)}$.

\begin{definition}
    Let $\Lag^{n, \qsp}_{(\infty, 1)}$ denote the full subcategory of
    $\Lag^n_{(\infty, 1)}$ where $p: M \to X$ is proper and $q: M \to Y$ is
    quasi-smooth.
\end{definition}

To this subcategory we can apply oriented Borel-Moore functors with
quasi-smooth pullbacks introduced in \cite{lowrey} and improved on to a
bivariant formalism in \cite{annala}. 

We briefly recall the definition of such a functor.

\begin{definition}
    An \emph{oriented Borel-Moore functor with quasi-smooth pullbacks} is a
    functor $A_*$ from derived schemes to graded abelian groups satisfying the
    axioms of an oriented Borel-Moore functor with products of
    \cite{levinemorel} along with homomorphisms $f^* : A_*(Y) \to A_{*+d}(X)$ of
    graded abelian groups for each equi-dimensional quasi-smooth homomorphism
    $f: X \to Y$. This functor must further satisfy the axioms listed in \cite{lowrey}.
\end{definition}

\begin{remark}
    In particular, such a functor has push-forward operations along proper
    morphisms.
\end{remark}

Given such a functor, we immediately get a good theory of virtual fundamental
classes. Given a quasi-smooth scheme $X$, we apply the quasi-smooth pullback
along the structure morphism $\pi_X: X \to \pt$ to obtain a virtual fundamental class
\[
    [X]^{\vir} := \pi_X^*([1])
\]
in the homology theory $A_*$.

We now want to do something similar with symplectic derived schemes. Ideally,
the virtual fundamental class of Oh and Thomas should again be the pull-back of
the class $[1]$ along a morphism.

To incorporate the symplectic form, we restrict our attention to a special kind
of oriented Borel-Moore functors.

\begin{definition}
    A \emph{symplectic oriented Borel-Moore functor with quasi-smooth pullbacks}
    is an oriented Borel-Moore functor with quasi-smooth pullbacks such that for
    any two Lagrangian correspondences in $\Lag^{n, \qsp}_{(\infty, 1)}$
    \[
        \begin{tikzcd}
            & M \ar[dl, "p" swap] \ar[dr, "q"] & \\
            X & & \pt 
        \end{tikzcd}
    \]
    and
    \[
        \begin{tikzcd}
            & M' \ar[dl, "p'" swap] \ar[dr, "q'"] & \\
            X & & \pt 
        \end{tikzcd}
    \]
    we have
    \[
        p_* q^*[1] = p'_* q'^*[1] \in A_*(X).
    \]

\end{definition}

\begin{remark}\label{rem:independent}
    This condition ensures the virtual fundamental class only depends on the
    symplectic form and is independent of the choice of Lagrangian
    correspondence.
\end{remark}

We now apply such a functor to the category
$\Lag^{n,\qsp}_{(\infty,1)}$.

\begin{definition}
    Let
    \[
        Z = 
        \begin{tikzcd}
            & M \ar[rd, "q"] \ar[ld, "p" swap] & \\
            X & & Y
        \end{tikzcd}
    \]
    be a Lagrangian correspondence with $p$ proper and $q$ quasi-smooth of
    relative virtual dimension $d$. We then define the operation
    \[
        Z^* : A_*(Y) \to A_{*+d}(X)
    \] as $Z^* = p_* \circ q^*$.
\end{definition}

In particular, given a quasi-smooth Lagrangian $M$ in a symplectid derived
scheme $(X, \omega)$ and a Lagrangian correspondence
\[
    \begin{tikzcd}
        & M \ar[rd, "q"] \ar[ld, "p" swap] & \\
        X & & \pt 
    \end{tikzcd}.
\]
we obtain a class
\[
    [(X,\omega)]^{\vir} = q_*p^*([1]) \in A_n(X).
\]
Using Remark \ref{rem:independent}, the class is independent of the choice of
Lagrangian and only depends on the symplectic form.

\begin{remark}
    An interesting question I could not find the answer to in the existing
    literature is if there always exist Lagrangians with respect to a derived
    scheme $X$ with shifted symplectic form $\omega$.
\end{remark}

With these definitions in place, we can show that the fundamental class arising
from a shifted symplectic derived structure pulls pack to the fundamental class.

\begin{thm}\label{thm:func}
    Let$[(Y,\sigma)]^{\vir}$ denote the virtual class represented by the
    quasi-smooth Lagrangian correspondece
    $\begin{tikzcd}[cramped, sep=small]
        Y & \ar[l, "j"] N \ar[r] & \pt
    \end{tikzcd}$
    in a symplectic oriented Borel-Moore homology theory $A_*$ with quasi-smooth
    pullbacks.  Assume there exists a quasi-smooth Lagrangian correspondence
    $\begin{tikzcd}[cramped,sep=small]
        X & \ar[l, "p"] \Gamma \ar[r, "q" swap] & Y
    \end{tikzcd}$
    . Then
    \[
        p_*q^* [(Y,\sigma)]^{\vir}= [(X,\omega)]^{\vir}.
    \]
\end{thm}
\begin{proof}
    We can compose the two Lagrangian correspondences to arrive at
    \[
        \begin{tikzcd}
            & & M \ar[ld, "f" swap] \ar[rd, "g"]& & \\
            & \Gamma \ar[rd, "q"] \ar[ld, "p" swap] & & N \ar[ld, "j" swap] \ar[rd, "\pi"] & \\
            X & & Y & & \pt
        \end{tikzcd}
    \]
    The outer diagram is again a Lagrangian correspondence, which gives us
    \[
        [X, \omega]^{\vir} = (p \circ f)_*(\pi \circ g)^*[1].
    \]
    By definition of the composition of Lagrangian correspondences, the upper
    square in the diagram is homotopy cartesian, so we have $q^*j_* = f_*g^*$.
\end{proof}

As a basic sanity check, we observe that symplectomorphisms preserve pullbacks
of virtual fundamental classes.

\begin{corollary}\label{cor:sanity}
    Let $[(X,\omega)]^{\vir}$ and $[(Y,\sigma)]^{\vir}$ be as above, and let $f\colon X
    \to Y$ be a symplectomorphism. Then
    \[
        f^* [(Y,\sigma)]^{\vir} = [(X,\omega)]^{\vir}.
    \]
\end{corollary}

\begin{proof}
    Use the diagram 
    \[
        \begin{tikzcd}
            & & M \ar[ld] \ar[rd]& & \\
            & X \ar[rd, "\id"] \ar[ld, "f" swap] & & N \ar[ld, "j" swap] \ar[rd, "\pi"] & \\
            X \ar[rr, "f"]& & Y & & \pt
        \end{tikzcd}
    \]
\end{proof}

In summary, we expect functoriality for any symplectic oriented Borel-Moore
homology theory.
It should be an interesting question to study if Chow homology is a symplectic
oriented Borel-Moore homology theory.

\section{Virtual Pullbacks of Oh-Thomas-classes}\label{hf}

In \cite{park}, Park studied the following question: Given two
quasi-projective schemes $X$ and $Y$ equipped with $(-2)$-symmetric
obstruction theories and a morphism $f$ equipped with a perfect
obstruction theory, when does the pull back of the Oh-Thomas virtual
fundamental class of $Y$ along $f$ map to the Oh-Thomas virtual
fundamental class of $X$?

Park identified that an additional complex is needed to ensure this is
the case, noted by $\mathbb{D}$ in his paper. We briefly recall the
necessary compatibilities he found.

Let $X$ and $Y$ denote quasi-projective (underived) schemes equipped
with $(-2)$-symmetric obstruction theories $\phi_X \colon \mathbb{E}_X
\to L_X$ and $\phi_Y \colon \mathbb{E_Y} \to L_Y$. Assume further we
have a morphism $f \colon X \to Y$ equipped with a relative perfect
obstruction theory $\phi_f \colon \mathbb{E}_f \to L_f$.

Given this relative perfect obstruction theory, we can perform the virtual
pullback $f^{!}$ defined by \cite{manolache}. This will not map the Oh-Thomas
virtual fundamental class of $Y$ to $X$. Instead, we have to add 
an additional complex $\mathbb{D}$ fitting into a compatibility diagram

\[
    \begin{tikzcd}
        \mathbb{D}^{\vee}[2] \arrow[r] \arrow[d] &
        \mathbb{E}_X \arrow[r] \arrow[d] & \mathbb{E}_f \arrow[d] \\
        f^* \mathbb{E}_Y \arrow[r] \arrow[d] & \mathbb{D} \arrow[r] \arrow[d] &
        \mathbb{E}_f \arrow[d]\\ f^* L_Y \arrow[r] & L_X \arrow[r] & L_f \\
    \end{tikzcd}
\]

In this diagram the horizontal rows are cofibre sequences.

We now want to give the derived geometric interpretation of the complex
$\mathbb{D}$ promised in the introduction using Lagrangian correspondences.
Motivated by Theorem \ref{thm:func}, we hope to find a Lagrangian
correspondence from $X \to Y$ which allows us to relate the classes.

\begin{lemma}\label{lem:comp}
    Let $X$ and $Y$ be $-2$-symplectic derived
    schemes. Assume that we have a Lagrangian correspondence
    \[
        \begin{tikzcd}
            & \Gamma \ar[rd, "q"] \ar[ld, "p" swap] & \\
            X & & Y
        \end{tikzcd}
    \].
    Then we have the following diagram of cotangent complexes
    \[
        \begin{tikzcd}
            T_\Gamma[2] \arrow{r}{\alpha^{\vee}[2]} \arrow{d}{\beta^{\vee}[2]
            swap} & p^*
            L_X \arrow{r} \arrow{d}{\alpha} & L_q \arrow[d, "\phi"]\\
            q^* L_{Y} \arrow{r}{\beta} & L_{\Gamma} \arrow[r] & L_q 
        \end{tikzcd}
    \]
    where the rows are cofibre sequences and $\phi$ is an equivalence.
\end{lemma}
\begin{proof}
    Since $\Gamma$ is a Lagrangian in $X \times Y$, we have by definition a
    cartesian and co-cartesian square
    \[
        \begin{tikzcd}
            T_{\Gamma} \ar[r] \ar[d] & p^* T_X \ar[d] \\
            q^* T_Y \ar[r] & L_{\Gamma}[-2]
        \end{tikzcd}
    \]
    Shifting the diagram by 2 and using the $(-2)$-symplectic structures on $X$
    and $Y$, we arrive at the left square of the diagram we
    wish to obtain. Since the diagram is cocartesian, the homotopy cofibres are
    equivalent, giving us the equivalence $\phi$. To conclude, we observe that
    by the Jacobi-Zariski sequence of the morphism $q$ we know that $\cofib(\beta) =
    L_q$.
\end{proof}

This already looks like the diagram identified by Park in \cite{park}.
The missing gadgets are the morphism $f \colon X \to Y$ and its cotangent
complex $L_f$. To bring these into play we need one further definition. 
To motivate the definition, recall from Section \ref{corr} that our goal is to
find a suitable Lagrangian correspondence relating the virtual fundamental
classes.

\begin{definition}\label{def:lagrangiangraph}
    Let $f \colon X \to Y$ be a morphism of $n$-shifted symplectic derived
    schemes $(X, \omega_X)$ and $(Y, \omega_Y)$. We call $\Gamma$ a Lagrangian
    graph with \'etale section for $f$ if there exists a Lagrangian correspondence
    \[
        \begin{tikzcd}
            & \Gamma \ar[rd, "q"] \ar[ld, "p" ] & \\
            X \ar[ur, bend left, "s"] \ar[rr, "f"]& & Y
        \end{tikzcd}
    \]
    such that $L_s \simeq 0$ and $p \circ s = \id$. 
\end{definition}

\begin{remark}
    This definition does not fit exactly with the conditions found in
    Theorem \ref{thm:func}. In particular, the \'etale morphism $s$ does
    not show up there. I was unable to make a better connection between the two
    approaches, although I assume there is a more general result for arbitrary
    Lagrangian correspondence.
    
    For instance Park can prove the above result without assuming a map from $X
    \to Y$ of derived schemes, only on the level of classical truncations.
    He considers Lagrangian correspondences
    $
    \begin{tikzcd}[cramped,sep=small]
        X & M \ar[r] \ar[l] & Y
    \end{tikzcd}
    $
    such that the classical truncation of $L \to X$ is an isomorphsim and $L \to
    Y$ is quasi-smooth, along with the existence of a compatible morphism $t_0(X) \to
    t_0(Y)$.
\end{remark}

\begin{corollary}\label{cor:comp}
    Assume $f \colon X \to Y$ is a morphism of $(-2)$-symplectic derived
    schemes, and that there exists a Lagrangian graph with \'etale section for $f$. We then have the
    following diagram of complexes on $X$
    \[
        \begin{tikzcd}
            T_\Gamma[2] \arrow{r}{\alpha^{\vee}[2]} \arrow{d}{\beta^{\vee}[2]} & p^*
            L_X \arrow{r} \arrow{d}{\alpha} & L_q \arrow[d, "\phi"]\\
            q^* L_{Y} \arrow{r}{\beta} & L_{\Gamma} \arrow[r] & L_q 
        \end{tikzcd}
    \]
    where the horizontal rows are cofibre sequences.
\end{corollary}
\begin{proof}
    Pull back the diagram of \ref{lem:comp} to $X$ using $s$ and observe $s^*
    L_q \simeq L_f$.
\end{proof}

We now have everything in place to show that morphisms of derived schemes which
have a Lagrangian graph with \'etale section give nice pullback operations for
virtual fundamental classes.

\begin{thm}
    Let $f \colon X \to Y$ be a quasi-smooth morphism of $(-2)$-symplectic
    derived schemes, and assume there exists a Lagrangian graph with \'etale
    section for $f$.  Let $ t_0(f) \colon t_0(X) \to t_0(Y) $ be the underlying
    morphism of classical schemes. Then
    \[
        t_0(f)^* [Y, \sigma]^{\vir} = [X, \omega]^{\vir} \in A_n(t_0(X), \mathbb{Z}[\frac 1 2])
    \]
\end{thm}
\begin{proof}
    Combing Lemma \ref{lem:comp} and Corollary \ref{cor:comp}, we obtain Park's
    compatibility diagram on $X$. Since $f$ is assumed to be quasi-smooth, the
    complex $L_f$ is relative perfect obstruction theory for $t_0(f)$.
\end{proof}

As a basic sanity check, we again verify that symplectomorphisms preserve
virtual fundamental classes, as in Corollary \ref{cor:sanity}.

\begin{corollary}
   Let $f \colon (X, \omega) \to (Y, \sigma)$ be a symplectomorphism of
   $-2$-symplectic derived schemes. Then
   \[
       f^*([Y, \sigma]^{\vir}) = [X, \omega]^{\vir} \in A_n(t_0(X), \mathbb{Z}[\frac 1 2])
   \]
\end{corollary}
\bibliographystyle{plain}
\bibliography{mybib}
\end{document}